\documentclass[a4paper,10pt]{amsart}

\usepackage{amssymb,amsthm,amsmath,verbatim}
\usepackage{t1enc}
\usepackage{enumerate}
\usepackage{color}
\usepackage{tikz}
\usepackage{breqn}
\usepackage{xmpmulti}
\usepackage{psfrag}

 \usepackage[usenames,dvipsnames]{pstricks}
 \usepackage{epsfig}
\usepackage{pst-grad}
\usepackage{pstricks,pst-node}
\usepackage{float}

\usepackage[left=3.5 cm, right=3.5cm]{geometry}
\numberwithin{equation}{section}

\newtheorem{prop}{Proposition}[section]
\newtheorem{dfn}[prop]{Definition}
\newtheorem{lemma}[prop]{Lemma}

\newtheorem{cor}[prop]{Corollary}
\newtheorem{thm}[prop]{Theorem}

\newtheorem{clm}[prop]{Claim}
\newtheorem{obs}[prop]{Observation}

\newtheorem{prob}[prop]{Problem}

\newcommand{\mc}[1]{\mathcal{#1}}
\newcommand{\mb}[1]{\mathbb{#1}}

\newcommand{\oo}{\omega}
\newcommand{\uhr}{\upharpoonright}
\newcommand{\omg}{{\omega_1}}

\DeclareMathOperator{\acc}{acc}
\DeclareMathOperator{\non}{non}

\DeclareMathOperator{\cf}{cf}

\def\<{\left\langle}
\def\>{\right\rangle}
\def\br#1;#2;{\bigl[ {#1} \bigr]^ {#2} }

\newcommand{\mf}[1]{\mathfrak{#1}}

\newcommand{\setm}{\setminus}

\newcommand{\subs}{\subset}

\newcommand{\ran}{\operatorname{ran}}

\title{More ZFC inequalities between cardinal invariants}

\date{\today}

  \author{Vera Fischer}
 \address[V. Fischer]{Universit\"at Wien,
Kurt G\"odel Research Center for Mathematical Logic, Wien, Austria}
 \email{vera.fischer@univie.ac.at}
  \urladdr{http://www.logic.univie.ac.at/$\sim $vfischer/}

  \author{D\'aniel T. Soukup}
  \address[D.T. Soukup]{Universit\"at Wien,
Kurt G\"odel Research Center for Mathematical Logic, Wien, Austria}
 \email[Corresponding author]{daniel.soukup@univie.ac.at}
 \urladdr{http://www.logic.univie.ac.at/$\sim  $soukupd73/}

\makeatletter
\newtheorem*{rep@theorem}{\rep@title}
\newcommand{\newreptheorem}[2]{%
\newenvironment{rep#1}[1]{%
 \def\rep@title{#2 \ref{##1}}%
 \begin{rep@theorem}}%
 {\end{rep@theorem}}}
\makeatother

\newreptheorem{corollary}{Corollary}
\newreptheorem{theorem}{Theorem}


\begin{document}
 \begin{abstract}  
 Motivated by recent results and questions of D. Raghavan and S. Shelah, we present ZFC theorems on the bounding and various almost disjointness numbers, as well as on reaping and dominating families on uncountable, regular cardinals.  We show that if $\kappa=\lambda^+$ for some $\lambda\geq \oo$ and $\mf b(\kappa)=\kappa^+$ then $\mf a_e(\kappa)=\mf a_p(\kappa)=\kappa^+$.  If, additionally, $2^{<\lambda}=\lambda$ then $\mf a_g(\kappa)=\kappa^+$ as well. Furthermore, we prove a variety of new bounds for $\mf d(\kappa)$ in terms of $\mf r(\kappa)$, including $\mf d(\kappa)\leq \mf r_\sigma(\kappa)\leq \cf([\mf r(\kappa)]^\oo)$, and $\mf d(\kappa)\leq \mf r(\kappa)$ whenever  $\mf r(\kappa)<\mf b(\kappa)^{+\kappa}$ or $\cf(\mf r(\kappa))\leq \kappa$ holds.
 

\end{abstract}
\maketitle

\section{Introduction}

In this article, we focus on unbounded, dominating and eventually different families of functions in $\kappa^\kappa$, and unsplit families of sets from $[\kappa]^\kappa$ for an uncountable, regular cardinal $\kappa$. There is a great history of such studies for $\kappa=\aleph_0$, which  later sparked significant interest in the case of $\kappa>\aleph_0$. Especially so, that some long unresolved questions for cardinal characteristics on $\aleph_0$ have been answered for uncountable cardinals (e.g. Roitman's problem \cite{blass}). The goal of our paper is to present new ZFC relations between cardinal invariants on an uncountable, regular $\kappa$, many that fail to hold in the countable case.

First, in Section \ref{sec:ub}, we show that if $\kappa$ is a successor cardinal and there is a $\leq^*$-unbounded family of functions in $\kappa^\kappa$ of size $\kappa^+$, then there is a maximal family of eventually different functions/permutations of size $\kappa^+$ as well. If, additionally, $2^{<\lambda}=\lambda$ holds where $\kappa=\lambda^+$, then there is a maximal group of eventually different permutations of $\kappa$ of size $\kappa^+$ as well. These results generalize recent work of D. Raghavan and S. Shelah \cite{dilip2}, and provide strengthening of certain results from \cite{blass,hytt}.

Next, in Section \ref{sec:dr}, we bound the minimal size of a $\leq^*$-dominating family by the minimal size of an unsplit family under various conditions. Raghavan and Shelah proved that $\mf d(\kappa)\leq \mf r(\kappa)$ whenever $\kappa\geq \beth_\oo$. Our main result here is Theorem \ref{newbound} that provides a variety of new bounds for $\mf d(\kappa)$ in terms of $\mf r(\kappa)$, and a new characterization of the dominating number (see Corollary \ref{cor:char}, \ref{cor:alephn}, and \ref{cor:pcf}). In an independent argument, we next show that $\cf(\mf r(\kappa))\leq \kappa$ implies $\mf d(\kappa)\leq \mf r(\kappa)$ as well (see Theorem \ref{thm:cof}). 

Finally, we summarize the relations between these invariants in three diagrams, and end our article by emphasizing the most important open problems in the area. In particular, it remains open if $\mf d(\kappa)\leq \mf r(\kappa)$ holds for all uncountable, regular $\kappa$.

\medskip

We aimed our paper to be  self contained, and to collect most of the known results on related cardinal invariants. Let us also refer the new reader to A. Blass' \cite{char} as a classical reference on cardinal characteristics on $\aleph_0$.

\subsection*{Acknowledgments}The authors would like to thank the Austrian Science Fund (FWF) for their generous support through the FWF Grant Y1012-N35 (V. Fischer), and the FWF Grant I1921 (D. Soukup).

\section{Unbounded and mad families of functions} \label{sec:ub}

Let us start by recalling some well known definitions. The \emph{bounding number} $\mf b(\kappa)$ is the minimal size of a family $\mc F\subs \kappa^\kappa$ so that there is no single function $g\in \kappa^\kappa$ so that $\{\alpha<\kappa:g(\alpha)< f(\alpha)\}$ has size $<\kappa$ for all $f\in \mc F$. In other words, $\mc F$ is unbounded in the relation $\leq^*$ of almost everywhere dominance. We use the fact that $\mf b_{cl}(\kappa)=\mf b(\kappa)$ for any uncountable, regular $\kappa$ \cite{cummings}: there is $\mc F\subs \kappa^\kappa$ of size $\mf b(\kappa)$ so that for any $g\in \kappa^\kappa$ there is some $f\in \mc F$ with  $\{\alpha<\kappa:g(\alpha)< f(\alpha)\}$ stationary. I.e., $\mc F$ is $\leq_{cl}$-unbounded where $f\leq_{cl}g$ iff $\{\alpha<\kappa:f(\alpha)\leq g(\alpha)\}$ contains a club\footnote{I.e., closed and unbounded.	} subset of $\kappa$.

We also remind the reader of the usual almost disjointness numbers; in our context almost disjoint (eventually different) means that the intersection of the sets (functions) has size $<\kappa$.

\begin{enumerate}[(i)]
 \item $\mf a(\kappa)$ is the minimal size of a maximal almost disjoint family $\mc A\subs [\kappa]^\kappa$ that is of size $\geq \kappa$ (the latter rules out trivialities like $\mc A =\{\kappa\}$).
\item $\mf a_e(\kappa)$ is the minimal size of a maximal, eventually different family of functions in $\kappa^\kappa$.
\item  $\mf a_p(\kappa)$ is the minimal size of a maximal, eventually different family of functions in $S(\kappa)$, the set of bijective members of $\kappa^\kappa$.
\item $\mf a_g(\kappa)$ is the minimal size of an almost disjoint subgroup of $S(\kappa)$, that is maximal among such subgroups.
\end{enumerate}

\medskip

D. Raghavan and S. Shelah \cite{dilip} recently proved that $\mf b(\kappa)=\kappa^+$ implies $\mf a(\kappa)=\kappa^+$ for any regular, uncountable $\kappa$, by an elegant, and surprisingly elementary application of Fodor's pressing down lemma. Building on their momentum, we extend this result to related cardinal invariants on maximal families of eventually different functions and permutations (see \cite{blass,hytt} for a detailed background).

\begin{thm}
Suppose that $\kappa=\lambda^+$ for some $\lambda\geq \oo$ and $\mf b(\kappa)=\kappa^+$. Then $\mf a_e(\kappa)=\mf a_p(\kappa)=\kappa^+$. If, additionally, $2^{<\lambda}=\lambda$ then $\mf a_g(\kappa)=\kappa^+$ as well.
\end{thm}

This is a strengthening of \cite[Theorem 2.2]{blass}, where $\mf d(\kappa)=\kappa^+$ implies $\mf a_e(\kappa)=\kappa^+$ was proved for successor $\kappa$, and also of \cite[Theorem 4]{hytt} where  $\mf b(\kappa)=\kappa^+$ implies $\mf a_e(\kappa)=\kappa^+$ was proved using additional assumptions.


\begin{proof} Let $\{f_\delta:\delta<\kappa^+\}$ witness $\mf b_{cl}(\kappa)=\kappa^+$. Also, fix bijections $e_\delta:\kappa\to \delta$  where $\kappa\leq\delta<\kappa^+$ and bijections $d_\alpha:\alpha\to \lambda$ where  $\lambda\leq\alpha<\kappa$. The latter will allow us, given some $H\subseteq \alpha<\kappa$ with $|H|=\lambda$ and $\zeta<\lambda$, to select the $\zeta^{th}$ element of $H$ with respect to $d_\alpha$; that is, to pick $\beta\in H$ so that $d_\alpha(\beta)\cap d_\alpha[H]$ has order type $\zeta$.

\medskip

\subsection*{Let us start with $\mf a_e(\kappa)=\kappa^+$.} We will define functions $h_{\delta,\zeta}\in \kappa^\kappa$ for $\delta<\kappa^+,\zeta<\lambda$ that will form our maximal eventually different family.

We go by induction on $\delta<\kappa^+$. For each $\mu<\kappa$, let $\mb H_\delta(\mu)=\{h_{\delta',\zeta'}:\delta'\in \ran (e_\delta\uhr \mu),\zeta'<\lambda\}$. 
Note that $$H_\delta(\mu)= \{h(\mu):h\in \mb H_\delta(\mu)\}$$ has size $<\kappa$, so we can define $$f^*_\delta(\mu)=\max\{f_\delta(\mu),\min\{\alpha<\kappa:|\alpha\setm H_\delta(\mu)|=\lambda\}\}.$$

We define $h_{\delta,\zeta}(\mu)$ to be the $\zeta^{th}$ element of $ f^*_\delta(\mu)\setm H_\delta(\mu)$ with respect to $d_{f^*_\delta(\mu)}$.

\begin{clm}
   $\mb H=\{h_{\delta,\zeta}:\delta<\kappa^+,\zeta<\lambda\}\subs \kappa^\kappa$ is eventually different.
\end{clm}
\begin{proof} For a fixed $\delta$ and $\zeta<\zeta'<\lambda$, $h_{\delta,\zeta}(\mu)\neq h_{\delta,\zeta'}(\mu)$ by definition for all $\mu<\kappa$.

Given $\delta'<\delta$, and $\zeta,\zeta'<\lambda$, whenever $\delta'\in \ran (e_\delta\uhr \mu)$, then $h_{\delta',\zeta'}\in \mb H_\delta(\mu)$ and so  $h_{\delta',\zeta'}(\mu)\neq h_{\delta,\zeta}(\mu)$, since $h_{\delta,\zeta}\notin H_\delta(\mu)$.
\end{proof}

  \begin{clm}
   $\mb H$ is maximal.
 \end{clm}
\begin{proof}
Fix some $h\in \kappa^\kappa$, and find $\delta<\kappa^+$ so that $$S=\{\mu<\kappa:h(\mu)<f_\delta(\mu)\}$$ is stationary. Now, there is a stationary $S_0\subs S$ so that either
\begin{enumerate}
 \item $h(\mu)\in H_\delta(\mu)$ for all $\mu\in S_0$, or
 \item $h(\mu)\notin H_\delta(\mu)$ for all $\mu\in S_0$.
\end{enumerate}
In the first case, for each $\mu$, we can find some $\delta'=e_\delta(\eta_\mu)$ with $\eta_\mu<\mu$ and $\zeta'=\zeta'_\mu<\lambda$ so that $h(\mu)=h_{\delta',\zeta'}(\mu)$. In turn, by Fodor's lemma, we can find a stationary $S_1\subs S_0$ and single $\delta'=e_\delta(\eta)$ and $\zeta'<\lambda$ so that $h(\mu)=h_{\delta',\zeta'}(\mu)$ for all $\mu\in S_1$; hence, $h\cap h_{\delta',\zeta'}$ has size $\kappa$.

In the second case, $h(\mu)\in f^*_\delta(\mu)\setm H_\delta(\mu)$ must hold too, and so there is a $\zeta_\mu<\lambda$ so that $h(\mu)$ is the $\zeta_\mu^{th}$ element of $  f^*_\delta(\mu)\setm H_\delta(\mu)$ with respect to  $d_{f^*_\delta(\mu)}$. Again, we can find a single $\zeta<\lambda$ and stationary $S_1\subs S_0$ so that $\zeta_\mu=\zeta$ for all $\mu\in S_1$ and so  $h\cap h_{\delta,\zeta}$ has size $\kappa$.

\end{proof}

This shows that $\mb H$ is the desired maximal eventually different family.

\medskip

\subsection*{Now, we proceed with $\mf a_p(\kappa)=\kappa^+$.} 
We will modify the previous argument to ensure $h_{\delta,\zeta}\in S(\kappa)$ and to keep the family maximal in $S(\kappa)$. 
Let $\bar e=\{e_\delta:\kappa\leq\delta<\kappa^+\}$, $\bar{d}=\{ d_\alpha:\lambda\leq\alpha<\kappa\}$ ,
$\bar{f}=\{ f_\delta:\delta<\kappa^+\}$. We will need some elementary submodels: for each $\delta<\kappa^+$, we fix a continuous, increasing sequence of elementary submodels $\bar N^\delta=(N^\delta_\eta)_{\eta<\kappa}$ of some $H(\theta)$ so that
\begin{enumerate}[(i)]
\item $|N^\delta_\eta|=\lambda$, and $N^\delta_\eta\cap \kappa\in \kappa$,
 \item $\delta, \bar e,\bar d, \bar f\in N^\delta_\eta$,
 \item $\delta'\in N^\delta_\eta$ implies $\bar N^{\delta'}\in N^\delta_\eta$.
\end{enumerate}

Let $E_\delta=\{N^\delta_\eta\cap \kappa:\zeta<\kappa\}\cup \{0\}$ which is a club in $\kappa$.

Again, we proceed by induction on $\delta$, but use the notation $\mb H_\delta(\nu)$ and $H_\delta(\nu)$ with minor modifications: $\mb H_\delta(\nu)=\{h_{\delta',\zeta'}:\delta'\in \ran (e_\delta\uhr \mu),\zeta'<\lambda\}$ where $\mu=\sup(E_\delta\cap \nu)\leq \nu$, and $$H_\delta(\nu)= \{h(\nu):h\in \mb H_\delta(\nu)\}.$$ So $\mb H_\delta(\nu)=\mb H_\delta(\mu)$ but $H_\delta(\nu)$ and $H_\delta(\mu)$ are typically different (where $\mu=\sup(E_\delta\cap \nu)$).

We construct $h_{\delta,\zeta}$ for $\zeta<\lambda$ so that

\begin{enumerate}
 \item $h_{\delta,\zeta}\uhr [\mu,\mu^+)\in S([\mu,\mu^+))$ for any successive elements $\mu,\mu^+$ of $E_\delta$,
 \item\label{ev1} $h_{\delta,\zeta}\cap h_{\delta,\zeta'}=\emptyset$ for $\zeta'<\zeta$,
 \item\label{ev2} $h_{\delta,\zeta}(\nu)\in \kappa\setm H_\delta(\nu)$,
 \item\label{max} $h_{\delta,\zeta}(\mu)$ is the $\zeta^{th}$ element of $ f^*_\delta(\mu)\setm (H_\delta(\mu)\cup \mu)$ with respect to $d_{f^*_\delta(\mu)}$, where$$f^*_\delta(\mu)=\max\{f_\delta(\mu),\min\{\alpha<\kappa:|\alpha\setm (H_\delta(\mu)\cup \mu)|=\lambda\}\},$$
 \item\label{elem} $(h_{\delta,\zeta})_{\zeta<\lambda}$ is uniquely definable from $\bar N^\delta$.
\end{enumerate}

These conditions clearly ensure that $h_{\delta,\zeta}\in S(\kappa)$, and as before, the family $\{h_{\delta',\zeta'}:\delta'\leq \delta,\zeta'<\lambda\}$ remains eventually different by conditions (\ref{ev1}) and (\ref{ev2}). Maximality, just as before, follows from  condition (\ref{max}) and Fodor's lemma.

Let us show that we can actually construct functions with the above properties. Fix  successive elements $\mu<\mu^+$ of $E_\delta$, and we define $h_{\delta,\zeta}\uhr [\mu,\mu^+)\in S([\mu,\mu^+))$ by an induction in $\lambda$ steps.   We list all triples from $\lambda\times [\mu,\mu^+)\times 2$ as $(\zeta_\xi,\nu_\xi,i_\xi)$ for $\xi<\lambda$.

First of all, let $\mu=\kappa\cap N^\delta_\eta$ and $\mu^+=\kappa\cap N^\delta_{\eta+1}$; we will write $N$ for $N^\delta_{\eta+1}$ temporarily.
\begin{clm}
 $\mu^+\setm (H_\delta(\nu)\cup \mu)$ has size $\lambda$ for all $\nu\in \mu^+\setm \mu$.
\end{clm}
\begin{proof}

Note that $e_\delta\uhr \mu\in N$ and so $\ran (e_\delta\uhr \mu)$ is an element and subset of $N$. Furthermore, we can apply (\ref{elem}) to see that $\mb H_\delta(\mu)\in N$ and so $H_\delta(\nu)\in N$ for any $\nu<\mu^+$. Moreover, $N\models |H_\delta(\nu)|<\kappa$ so $\mu^+\setm (H_\delta(\nu)\cup \mu)$ has size $\lambda$.
\end{proof}

In turn, since $f_\delta(\mu)\in N$ as well, the value  $f^*_\delta(\mu)$ in condition (\ref{max}) is well defined and $<\mu^+$.

Now, we can start our induction on $\xi<\lambda$ by partial functions $h_{\delta,\zeta}$, each defined only at $\mu$ to satisfy condition (\ref{max}). At step $\xi$, we do the following. Let $\zeta=\zeta_\xi, \nu=\nu_\xi$; if $i_\xi=0$ then we make sure that $\nu$  gets into the domain of $h_{\delta,\zeta}$, and if  $i_\xi=1$ then we make sure that $\nu$ is in the range of $h_{\delta,\zeta}$.

Suppose $i_\xi=0$. We need to find a value for $h_{\delta,\zeta}(\nu)$ which is in $\mu^+\setm (H_\delta(\nu)\cup \mu)$ and which also avoids $h_{\delta,\zeta'}(\nu)$ where $\zeta'=\zeta_{\xi'}$ for some $\xi'<\xi$. The set   $\mu^+\setm (H_\delta(\nu)\cup \mu)$ has size $\lambda$ (using that $H_\delta(\nu)\in N$ as before), and we only defined $<\lambda$ many functions so far, hence we can find a (minimal) good choice.

Next, if $i_\xi=1$ then we need to find some $\vartheta\in \mu^+\setm \mu$ so that $h(\vartheta)\neq \nu$ for $h\in \mb H_\delta(\mu)$ and $h_{\delta,\zeta'}(\vartheta)\neq \nu$ for all $\zeta'=\zeta_{\xi'}$ for some $\xi'<\xi$. First, $\mb H_\delta(\mu)\in N$ and has size $<\kappa$ so the set of good choices $$\mu^+\setm (\mu \cup \{\vartheta<\kappa:h(\vartheta)=\nu,h\in \mb H_\delta(\mu)\}$$ still has size $\lambda$ by elementarity. Each $h_{\delta,\zeta'}$
 introduces $\leq 1$ bad $\vartheta$, and we have $\leq |\xi|<\lambda$ many of these, so we can find a good (minimal) $\vartheta$.

 If we carry out all this work in $N^\delta_{\eta+2}$, always taking minimal choices, then in the end condition (\ref{elem}) is preserved as well.

 \medskip

\subsection*{Finally, we turn to the proof of  $\mf a_g(\kappa)=\kappa^+$.} We use the additional assumption that $2^{<\lambda}=\lambda$. We keep the notations $\mb H_\delta(\nu),H_\delta(\nu)$ from the previous section, as well as the elementary submodels. However, we can now assume that each successor model $N^\delta_{\eta+1}$ is $<\lambda$-closed. This will help us when we are constructing the functions $h_{\delta,\zeta}$ in the induction of length $\lambda$, because at each intermediate step $\xi$, the model $N^\delta_{\xi+1}$ will contain all the functions which we constructed so far (there was no reason for this to hold before).

So, our aim now is to construct $\mb H = \{h_{\delta,\zeta}:\delta<\kappa^+,\zeta<\lambda\}\subs S(\kappa)$, so that in the generated subgroup $\mb G=\langle \mb H\rangle $, only the identity has $\kappa$ fixed points and  $\mb G$ is maximal. We use the notation $$\mb G_\delta(\nu)=\langle \mb H_\delta(\nu)\rangle \textmd{ and } G_\delta(\nu)=\{g(\nu):g\in \mb G_\delta(\nu)\}$$ for $\delta<\kappa^+$ and $\nu<\kappa$.

 We go by induction on $\delta$ as before, and  construct $h_{\delta,\zeta}$ so that
 \begin{enumerate}
 \item $h_{\delta,\zeta}\uhr [\mu,\mu^+)\in S([\mu,\mu^+))$ for any successive elements $\mu,\mu^+$ of $E_\delta$,
 \item\label{fix} any fixed point of a non identity function $h\in \langle \mb H_\delta(\mu)\cup \{h_{\delta,\zeta}\uhr \mu^+:\zeta<\lambda\}\rangle $ is below $\mu$,

 \item\label{elemm} $(h_{\delta,\zeta})_{\zeta<\lambda}$ is uniquely definable from $\bar N^\delta$.
\end{enumerate}
These conditions ensure that only the identity in $\mb G$ has $\kappa$ fixed points. Indeed, suppose $g\in \mb G$ is not the identity and write it as a finite product of $h_{\delta,\zeta}$ functions. Let $\delta_1$ be the maximal $\delta$ that occurs; if no other $\delta$ is in this product then $g$ has no fixed points by (\ref{fix}). If $\delta_0$ is the maximum of all other $\delta$'s that occur then we can find a $\mu<\kappa$ so that $\delta_0\in \ran(e_\delta\uhr \mu)$ and so (\ref{fix}) implies that all fixed points of $g$ are below $\mu$.

As before, we fixed some $\mu<\mu^+$, and $h_{\delta,\zeta}$ is constructed by an induction of length $\lambda$, using an enumeration of all triples from $\lambda\times [\mu,\mu^+)\times 2$ as $(\zeta_\xi,\nu_\xi,i_\xi)$ for $\xi<\lambda$.

We start by empty functions now, and at step $\xi$, we either need to put $\nu=\nu_\xi$ into the domain of $h_{\delta,\zeta}$  or into the range of $h_{\delta,\zeta}$ (where $\zeta=\zeta_\xi$).

Lets look at the first case: in order to preserve (\ref{fix}), it suffices to ensure that $h_{\delta,\zeta}(\nu)\neq h(\nu)$ for any $$h\in Z=\langle \mb H_\delta(\mu)\cup \{h_{\delta,\zeta'}:\zeta'=\zeta_{\xi'},\; \xi'<\xi\}\rangle$$ whenever $h(\nu)$ can be computed.

The maps $h_{\delta,\zeta'}$ are some partial functions on $\mu^+$ that extend $h_{\delta,\zeta'}\uhr \mu$ by $<\lambda$ many new values. Since $N=N^\delta_{\xi+1}$ now contains these functions as well as the set $\{h_{\delta,\zeta'}:\zeta'=\zeta_{\xi'},\; \xi'<\xi\}$,  it also contains the set $Z$ (we applied that $N$ is $<\lambda$-closed and the inductive hypothesis (\ref{elemm})). So, since $$N\models |\{h(\nu):h\in Z\}|<\kappa,$$ we can take  $h_{\delta,\zeta}(\nu)=\min \mu^+\setm(\{h(\nu):h\in Z\}\cup \mu)$.

  To ensure maximality in the end, we consider the case $\nu=\mu$ separately. Now, we don't just take a minimal good choice but look at the minimal $\alpha\geq f_\delta(\mu)$ so that $\alpha \setm (\{h(\mu):h\in Z\}\cup \mu)$ has size $\lambda$. Since $Z\in N$ and $N\models |Z|<\kappa$,  $\alpha\in N$ too. Now,  we define $h_{\delta,\zeta}(\mu)$ to be the $\zeta^{th}$ element of $\alpha \setm \{h(\mu):h\in Z\}$ with respect to $d_\alpha$.

Second, to put $\nu$ in the range of $h_{\delta,\zeta}$: we need some $\vartheta\in \mu^+\setm \mu$ so that $h(\vartheta)\neq \nu$ for any $h\in Z$ (and then we can set $h_{\delta,\zeta}(\theta)=\nu$). Again, $N\models |Z|<\kappa$ and each $h\in Z$ contributes with at most one bad $\vartheta$ so we can pick a minimal $\vartheta$ that works.

 It is left to check that we constructed a maximal $\mb G$. Fix any $g\in S(\kappa)\setm \mb G$ and find $\delta<\kappa^+$ so that $S=\{\mu<\kappa:g(\mu)<f_\delta(\mu)\}$ is stationary. Now, there is a stationary $S_0\subs S$ so that either
  \begin{enumerate}
   \item $g(\mu)=h(\mu)$ for some $\langle H_\delta(\mu)\cup \{h_{\delta,\zeta}:\zeta<\lambda\}\rangle$ for all $\mu\in S_0$, or
   \item $g(\mu)\neq h(\mu)$ for all $\langle H_\delta(\mu)\cup \{h_{\delta,\zeta}:\zeta<\lambda\}\rangle$ for all $\mu\in S_0$.
  \end{enumerate}
In the first case, we can use Fodor's theorem to fix a single $h\in \mb G$ so that $g\cap h$ has size $\kappa$. In the latter, there is some $\zeta<\lambda$ so that $g\cap h_{\delta,\zeta}$ has size $\kappa$ (just as in the previous proofs).

\end{proof}

We do not know at this point if our theorem is true without the assumption of $\kappa$ being successor, nor how to remove $2^{<\lambda}=\lambda$ from the last part of the result.

\section{Reaping and dominating numbers}\label{sec:dr}

\subsection*{Some background} Let us recall a few more invariants, first the dual of the bounding number:

\begin{enumerate}[(i)]
\item the \emph{dominating number} $\mf d(\kappa)$ is the minimal size of a family $\mc F\subs \kappa^\kappa$ which is $\leq^*$-dominating;
\item the \emph{club-dominating number} $\mf d_{cl}(\kappa)$ is the minimal size of a family $\mc F\subs \kappa^\kappa$ which is $\leq_{cl}$-dominating.\footnote{Recall that $f\leq_{cl} g$ if $\{\alpha<\kappa:f(\alpha)\leq g(\alpha)\}$ contains a club.}
\end{enumerate}

Second, we will look at the \emph{reaping and splitting numbers}. We say that $B$ \emph{splits} $A$ if $|A\cap B|=|A\setm B|=\kappa$. 

\begin{enumerate}[(i)]

\item $\mf r(\kappa)$ is the minimal size of a family $F\subs [\kappa]^\kappa$ so that no single $B\subs \kappa$ splits all $A\in F$;
\item $\mf s(\kappa)$ is the minimal size of a family $F\subs [\kappa]^\kappa$ so that any $A\in [\kappa]^\kappa$ is split by some $B\in F$.
\end{enumerate}

It was proved by Cummings and Shelah that $$\mf d_{cl}(\kappa)\leq \mf d(\kappa)\leq \cf([\mf d_{cl}(\kappa)]^\oo),$$ and $\mf d(\kappa)= \mf d_{cl}(\kappa)$ whenever $\kappa\geq \beth_\oo$ \cite{cummings}. It is not known if the latter assumption is necessary.

\medskip

To the surprise of many, Raghavan and Shelah \cite{dilip2} proved that $\mf s(\kappa)\leq \mf b(\kappa)$ for any uncountable, regular $\kappa$ (a result that consistently fails for $\kappa=\oo$).\footnote{Even before this result, it was known that $\mf s(\kappa)$ behaves very interestingly for an uncountable $\kappa$. $\mf s(\kappa)\geq \kappa$ iff $\kappa$ is weakly inaccessible, $\mf s(\kappa)> \kappa$ iff $\kappa$ is weakly compact, and  $\mf s(\kappa)>\kappa^+$ is equiconsistent with the existence of a measurable cardinal $\mu$ with Mitchell order at least $\mu^{++}$ \cite{omer, suzuki, zapletal}.} 
In fact, they prove that $$\mf s(\kappa)\leq\mf p_{cl}(\kappa)\leq \mf b(\kappa)$$ where $\mf p_{cl}(\kappa)$ is the minimal size of a family of clubs $\mc D$ in $\kappa$ without a pseudo intersection i.e., no $A\in [\kappa]^\kappa$ satisfies $A\subseteq^* D$ for all $D\in \mc D$. 
We remark that $$\kappa^+\leq \mf p(\kappa)\leq \mf t(\kappa) \leq \mf b(\kappa)$$ where $\mf p(\kappa)$ is the minimal size of a family with the $\kappa$-intersection property (i.e., any subfamily of size $<\kappa$ has an intersection of size $\kappa$) but without a pseudo intersection, and $\mf t (\kappa)$ is the minimal size of a $\subseteq^*$-chain with the $\kappa$-intersection property in $[\kappa]^\kappa$ without a pseudo intersection \cite{brook}.\footnote{S. Garti proved that $\mf p(\kappa)=\mf t (\kappa)$ if $\mf p(\kappa)=\kappa^+$ and $\kappa^{<\kappa}=\kappa$ \cite{garti}. Given the recent breakthrough of M. Malliaris and Shelah \cite{malliaris}  proving $\mf p(\oo)=\mf t(\oo)$, it would be interesting to see how much of that machienary can be generalized to uncountable cardinals.} While $\mf p(\kappa)\leq \mf p_{cl}(\kappa)$ clearly holds by definition, we are not aware of results separating these invariants.

\medskip

Most recently, Raghavan and Shelah \cite{dilip} showed the dual inequality $$\mf d(\kappa)\leq \mf r(\kappa)$$ whenever $\kappa\geq \beth_\oo$, and this is where our interest lies. Especially so, that it is not known at this point if the assumption  $\kappa\geq \beth_\oo$ can be removed from their result. 

 Raghavan and Shelah's argument is surprisingly short, and goes as follows. For a set $E\subs \kappa$ and $\xi\in \kappa$, we let $s_E(\xi)=\min E\setm(\xi+1)$. 
 
 Take an arbitrary $F\subs [\kappa]^\kappa$. First, if there is a club $E_1$ so that for any club $E_2\subs E_1$ there is some $A\in F$ so that $A\subs^* \bigcup_{\xi\in E_2}[\xi,s_{E_1}(\xi))$ then $\mf d(\kappa)\leq |F|$. Indeed, the functions $g_A(\xi)=s_A(s_{E_1}(\xi))$ for $A\in F$ must $\leq^*$-dominate.

So, suppose that $F$ has size $\mf r(\kappa)$, and by the previous observation, we can assume the following: for any club $E_1$, there is a club  $E_2\subs E_1$ so that $A\not \subs^* \bigcup_{\xi\in E_2}[\xi,s_{E_1}(\xi))$  for all $A\in F$. In this case we say that $F$ has \emph{property RS} (for Raghavan-Shelah). Let us emphasize this definition.

\begin{dfn}We say that  $F\subs [\kappa]^\kappa$ has property RS if for any club $E_1$, there is a club  $E_2\subs E_1$ so that $A\not \subs^* \bigcup_{\xi\in E_2}[\xi,s_{E_1}(\xi))$  for all $A\in F$.
\end{dfn}

Now, the next observation follows by definition.

\begin{obs}
If $F$ is unsplit and has property RS then for any club $E_1\subs \kappa$, there is a club  $E_2\subs E_1$  and some $A\in F$ so that $A\subs^* \kappa \setm \bigcup_{\xi\in E_2}[\xi,s_{E_1}(\xi))$.
\end{obs}
\begin{proof}

Indeed, given $E_1$ we find $E_2$ using property RS so that $A\not \subs^* \bigcup_{\xi\in E_2}[\xi,s_{E_1}(\xi))$  for all $A\in F$. If $A\cap \bigcup_{\xi\in E_2}[\xi,s_{E_1}(\xi))$ has size $\kappa$ for all $A\in F$ then $\bigcup_{\xi\in E_2}[\xi,s_{E_1}(\xi))$ would split $F$, which contradicts that $F$ is unsplit.
 
\end{proof}

Now, we claim that $\{s_A:A\in F\}$ is $\leq_{cl}$-dominating for an unsplit RS-family. Given $f\in \kappa^\kappa$, take an $f$-closed club $E_1$ and find $E_2\subs E_1$ and $A\in F$ using property RS so that $A\setm \delta \subs \kappa \setm \bigcup_{\xi\in E_2}[\xi,s_{E_1}(\xi))$ for some $\delta<\kappa$. Now, we claim that $f\uhr E_2\setm \delta\leq s_A\uhr E_2\setm \delta$. Indeed, $\xi\in E_2\setm \delta$ implies that $f(\xi)<s_{E_1}(\xi)\leq s_A(\xi)$ since $s_A(\xi)\in A$ and $A\cap [\xi,s_{E_1}(\xi))=\emptyset$.

This proves $$\mf d_{cl}(\kappa)\leq \mf r(\kappa),$$ and in turn,  $\mf d(\kappa)=\mf d_{cl}(\kappa)\leq \mf r(\kappa)$ follows if $\kappa\geq \beth_\oo$ by the Cummings-Shelah result above.
\medskip

\subsection*{New results} Recall that $\mf r_\sigma(\kappa)$ is the minimal size of a family $F\subs [\kappa]^\kappa$ so that there is no countable family $\{B_n:n<\oo\}$ so that any $A\in F$ is split by some $B_n$. It is easy to see that $\mf r_\sigma(\aleph_0)$ exists, however this is not so obvious for an uncountable $\kappa$.

\begin{obs}\cite[Lemma 3]{zapletal}
 If $\aleph_0<\kappa\leq 2^{\aleph_0}$ then there is a countable $\mc B$ that splits all $A\in [\kappa]^\kappa$.
\end{obs}
In turn, $\mf r_\sigma(\kappa)$ does not exist if  $\aleph_0<\kappa\leq 2^{\aleph_0}$. 

\begin{proof} Take an injection $f:\kappa\to 2^\oo$ and let $B_s=\{\alpha<\kappa:s\subs f(\alpha)\}$ for $s\in 2^{<\oo}$. We claim that $\{B_s:s\in 2^{<\oo}\}$ splits all $A\in [\kappa]^\kappa$. Indeed, this follows from the fact that any uncountable set of reals has at least two complete accumulation points. In detail, assume that some $A\subs \kappa$ is not split by any $B_s$. Then the set $S=\{s\in 2^{<\oo}:|A\cap B_s|=\kappa\}$ cannot contain incompatible elements (as $B_s\cap B_t=\emptyset$ whenever $s,t$ are incompatible), and so there is at most one $\alpha$ such that $s\in S$ implies $s\subs f(\alpha)$. In turn, $$A\subs \{\alpha\}\cup \bigcup_{s\in 2^{<\oo}\setm S} A\cap B_s$$ and the latter set has size $<\kappa$.

 \end{proof}

\begin{prop}
 If $\cf(\kappa)>2^{\aleph_0}$ then $\mf r_\sigma(\kappa)$ exists, and $\mf r(\kappa)\leq \mf r_\sigma(\kappa) \leq \cf([\mf r(\kappa)]^\oo)$.
\end{prop}
\begin{proof}
 Given a countable family $\mc B=\{B_n:n<\oo\}$, we can look at the map $g_{\mc B}:\kappa\to [\oo]^\oo$ defined by $g_{\mc B}(\alpha)=\{n\in \oo:\alpha\in B_n\}$. For any $A\in [\kappa]^\kappa$, it is equivalent that 
 \begin{enumerate}
  \item  no element of $\mc B$ splits $A$, and 
  \item $g_{\mc B}$ is eventually constant on $A$. 
 \end{enumerate}
  Suppose that $F\subs [\kappa]^\kappa$ is a reaping family of size $\mf r(\kappa)$, and $\{\mc B_\xi:\xi<\lambda\}$ is cofinal in $[\mf r(\kappa)]^\oo$ of size $\lambda=\cf([\mf r(\kappa)]^\oo)$. Find $A_\xi\in [\kappa]^\kappa$ so that $g_{\mc B_\xi}\uhr A_\xi$ is constant, which  can be done by $\cf(\kappa)>2^{\aleph_0}$. 
 
 We claim that $\{A_\xi:\xi<\lambda\}$ cannot be split by any countable family $\mc B$, and so $\mf r_\sigma(\kappa) \leq \cf([\mf r(\kappa)]^\oo)$. Indeed, given $\mc B$, find $\xi<\lambda$ so that $\mc B\subs \mc B_\xi$, and so $A_\xi$ is not split by any member of $\mc B$ (since $g_{\mc B}$ is constant on $A_\xi$).

 Note that once $\mf r_\sigma(\kappa)$ exists, $\mf r(\kappa)\leq \mf r_\sigma(\kappa)$ trivially holds by definition.
 \end{proof}

For $\kappa=\oo$, we know that  $\mf r_\sigma\leq \max\{\cf([\mf r]^\oo),\non(\mc M)\}$  (see \cite[Theorem 3.6]{brendlesplitting}), and it is a long standing open problem whether $\mf r(\oo)<\mf r_\sigma(\oo)$ is possible \cite{brendlesplitting}.

\begin{prop}\label{prop:rsigma}
$\mf d(\kappa)\leq \mf r_\sigma(\kappa)$ for any uncountable, regular $\kappa> 2^{\aleph_0}$.
\end{prop}

\begin{proof} Take a family $F$ which is not split by countably many sets and has size $\mf r_\sigma(\kappa)$. Again, we can suppose that $F$ is RS (otherwise $\mf d(\kappa)\leq |F|$ holds); we will show that $\{s_A:A\in F\}$ is dominating. Pick any $f\in \kappa^\kappa$, and we may assume $\alpha\leq f(\alpha)$ for all $\alpha<\kappa$. Starting form an $f$-closed club $E_0$, build $E_0\supseteq E_1\supseteq \dots$ clubs in $\kappa$ so that $$A\not \subs^* B_n=\bigcup_{\xi\in E_{n+1}}[\xi,s_{E_{n}}(\xi))$$  for all $A\in F$ and $n\in \oo$ (this is simply by applying that $F$ is RS inductively). The family $\{B_n\}_{n\in \oo}$ cannot split $F$ so there is a single $A\in F$ unsplit by all the $B_n$. This means that $$A\setm \delta \subs \kappa\setm B_n$$  some $\delta<\kappa$ and for all $n<\oo$. We claim that $f\leq^* s_A$.

Indeed, for any $\alpha\in \kappa\setm \delta$, we can find $n<\oo$ so that $\sup (E_n\cap (\alpha+1))=\sup (E_{n+1}\cap(\alpha+1))$ and let $\xi$ denote this common value. Now $$\xi\leq \alpha\leq f(\alpha)<s_{E_n}(\xi)<s_A(\alpha)$$ as desired.
 
\end{proof}

Next, we use this result to present a new characterization of $\mf d(\kappa)$ for uncountable $\kappa$. For $\kappa=\oo$, the value of $\min\{\mf r,\mf d\}$ is actually equal to the minimal size of a family of partitions $\mc I$ of $\oo$ into finite sets (equivalently, partitions to intervals) so that there is no single $A\in [\oo]^\oo $ that splits all $(I_n)_{n\in \oo}\in \mc I$ in the sense that both $\{n\in \oo:I_n\subs A\}$ and $\{n\in \oo:I_n\cap A=\emptyset\}$ are infinite. This invariant, the \emph{finitely reaping number}, is denoted by $\mf{fr}$ \cite{brendlesplitting}.

Now, the equivalent of this invariant for an uncountable and regular $\kappa$, which we denote by $\mf{fr}(\kappa)$, is the following: the minimal size of a family of clubs $\mc E$ so that  there is no single $A\subs \kappa$ such that both $\{\xi\in E:[\xi,s_E(\xi))\subs A\}$ and $\{\xi\in E:[\xi,s_E(\xi))\cap A=\emptyset\}$ have size $\kappa$ for all $E\in \mc E$. We say that $A$ \emph{interval-splits} $E$. It is easily shown, just like the above cited \cite[Proposition 3.1]{brendlesplitting}, that  $\mf{fr}(\kappa)=\min\{\mf d(\kappa),\mf r(\kappa)\}$ still holds.

Similarly, one proves that  $\mf{fr}_\sigma(\kappa)=\min\{\mf d(\kappa),\mf r_\sigma(\kappa)\}$, and so we actually get  $\mf{fr}_\sigma(\kappa)=\mf d(\kappa)$ by Proposition \ref{prop:rsigma} for an uncountable, regular $\kappa$. In other words:

\begin{cor}\label{cor:char} Suppose that $\kappa$ is regular and uncountable. Then $\mf d(\kappa)$ is the minimal size of a family of clubs $\mc E$ so that  there is no countable family $\mc A$ of subsets of $\kappa$ such that both $\{\xi\in E:[\xi,s_E(\xi))\subs A\}$ and $\{\xi\in E:[\xi,s_E(\xi))\cap A=\emptyset\}$ have size $\kappa$ for all $E\in \mc E$ and some $A\in \mc A$.
\end{cor}

For the sake of completeness, we sketch the argument:

\begin{proof}
 First, we prove  $\mf{fr}_\sigma(\kappa)\leq \mf d(\kappa)$: given a dominating family $\mc F\subs \kappa^\kappa$, take some $f$-closed club $E_f$  for each $f\in \mc F$ and let $\mc E=\{E_f:f\in \mc F\}$. We claim that there is no countable family $\mc A$ such that each $E_f$ is split by some $A\in \mc A$. Indeed, let $g=\sup\{s_A:A\in \mc A\}$ and find $f\in \mc F$ so that $g\leq ^* f$. It is easy to see that almost all intervals of $E_f$ meet all $A\in \mc A$.

 Now, suppose that we are given a family of clubs $\mc E$ of size $<\mf d(\kappa)$. First, we can find a single $f\in \kappa^\kappa$ so that $\{\alpha<\kappa:s_E\circ s_E(\alpha)<f(\alpha)\}$ has size $\kappa$ for all $E\in \mc E$. So, if $D$ is an $f$-closed club in $\kappa$ then $X_E=\{\zeta\in D:[\xi,s_E(\xi))\subs [\zeta,s_D(\zeta))$ for some $\xi\in E\}$ has size $\kappa$. Since  $|\mc E|<\mf d(\kappa)\leq \mf r_\sigma(\kappa)$, there is a countable family  $\{B_n:n\in \oo\}$ so that any $X_E$ is split by some $B_n$. So, we define $$A_n=\bigcup \{[\zeta,s_D(\zeta)):\zeta\in D\cap B_n\}$$ for $n<\oo$. Now, any $E\in \mc E$ must be interval-split by some element of $\{A_n:n\in \oo\}$.
 
\end{proof}

Returning to the question whether $\mf d(\kappa)\leq \mf r(\kappa)$ for any uncountable $\kappa$, we present the following new results.

\begin{thm}\label{newbound}Suppose that $\kappa$ is uncountable and regular. Then
\begin{enumerate}
\itemsep0.5em 
 \item\label{one} $\mf d(\kappa)\leq \sup_{\lambda<\mf r(\kappa)}\cf([\lambda]^\oo)\leq \cf([\mf r(\kappa)]^\oo)\leq \mf r(\kappa)^\oo$, 
 \item\label{two} $\mf d(\kappa)\leq \cf([\mf r(\kappa)]^\theta)$ for any $\oo\leq \theta<\mf b(\kappa)$, and
 \item\label{three} if $\mf r(\kappa)<\mf b(\kappa)^{+\kappa}$ then $\mf d(\kappa)\leq \mf r(\kappa)$.
\end{enumerate}
 \end{thm}

In fact, the proof of Theorem \ref{newbound} will follow from a closer analysis of Raghavan and Shelah's arguments in \cite{dilip2}, given in the next lemma.


\begin{lemma}\label{lm:tech} Suppose that $\kappa$ is uncountable and regular, and $F\subs [\kappa]^\kappa$ is an RS family. Also, assume that $\mc G\subs \kappa^\kappa$ and for any $F_0\in [F]^\kappa$ there is some $g\in \mc G$ such that $$|\{A\in F_0:s_A\leq^*g\}|\geq \aleph_0.$$ Then $\mc G$ is $\leq^*$-dominating in $\kappa^\kappa$ as well. 
\end{lemma}

First, let us show how the theorem is deduced from this lemma.

\begin{proof}[Proof of Theorem \ref{newbound}] Assume that $F=\{A_\xi:\xi<\mu\}\subs [\kappa]^\kappa$ is an RS family of size $\mu=\mf r(\kappa)$, and let  $F\uhr \lambda=\{A_\xi:\xi<\lambda\}$ for $\lambda<\mu$. Observe that for any $F_0\in [F]^\kappa$ there is some $\lambda<\mu$ so that $F_0\cap F\uhr \lambda$ is infinite.

(\ref{one})  Suppose that $\mc B_\lambda\subs [F\uhr \lambda]^\oo$ is cofinal in $[F\uhr \lambda]^\oo$ for $\lambda<\mu$. If now $F_0\in [F]^\kappa$ then there is some $\lambda<\mu$ and $X\in \mc B_\lambda$ so that $F_0\cap X$ is infinite. Hence, if we let $g_X\in \kappa^\kappa$ dominate $\{s_A: A\in X\}$, then $\mc G=\{g_X:X\in \mc B_\lambda,\lambda<\mu\}$ satisfies the assumptions of Lemma \ref{lm:tech}. In turn, the first inequality of (1) holds (and the rest trivially follows).

(\ref{two}) As before, if $\mc B\subs [F]^\theta$ cofinal and $g_X\in \kappa^\kappa$ dominates $\{s_A: A\in X\}$ for $X\in \mc B$ then $\mc G=\{g_X:X\in \mc B\}$ satisfies the assumptions of Lemma \ref{lm:tech}.

(\ref{three}) We claim that for any $F\subs [\kappa]^\kappa$ of size $<\mf b(\kappa)^{+\kappa}$ there is some $\mc G$ of size $|F|$ that satisfies the assumptions of Lemma \ref{lm:tech}. We still use $\mu$ for the size of $F$ and keep the notation $F\uhr \lambda$.

Assume first that $\mu=\mf b(\kappa)$. We can find $\mc G=\{g_\lambda:\lambda<\mu\}$ so that $g_\lambda$ dominates $\{s_A:A\in F\uhr \lambda\}$. By the above observation, $\mc G$ satisfies the assumptions of Lemma \ref{lm:tech} and we are done.

In general, we proceed by induction. Suppose that the claim is proved for  $\mu=\mf b(\kappa)^{+\zeta}$ for $\zeta<\xi$ where $\xi<\kappa$. Since $\cf(\xi)<\kappa$, for any $F_0\in [F]^\kappa$ there is some $\lambda<\mu$ so that $|F_0\cap F\uhr \lambda|=\kappa$. So, if $\mc G_\lambda$ is the family provided by the inductive hypothesis that satisfies the assumptions of Lemma \ref{lm:tech} for $F\uhr \lambda$, then $\mc G=\cup\{\mc G_\lambda:\lambda<\mu\}$ works for $F$.

This, in turn, implies  $\mf d(\kappa)\leq \mf r(\kappa)$ whenever $\mf r(\kappa)<\mf b(\kappa)^{+\kappa}$ by Lemma \ref{lm:tech}.
\end{proof}

Now, we prove the lemma:

\begin{proof}[Proof of Lemma \ref{lm:tech}]
This statement is essentially proved in \cite{dilip}, but let us reiterate: fix an $f\in \kappa^\kappa$, and we construct clubs $E^i_2\subs E^i_1\subs \kappa$ and $A^i\in F$ for $i<\kappa$ so that $E^i_1\subs \bigcap_{j<i}E^j_2$ and $$A^i\setm \delta^i \subs \kappa \setm \bigcup_{\xi\in E^i_2}[\xi,s_{E^i_1}(\xi))$$ for some $\delta^i<\kappa$. Moreover, we assume that $E^0_0$ is $f$-closed.

Now, apply the assumption on $\mc G$ and the set $F_0=\{A^i:i<\kappa\}$ to find $g\in \mc G$ and increasing $i_0<i_1<i_2<\dots<\kappa$  so that $s_{A^{i_n}}\leq^* g$ for all $n<\oo$. We claim that $f\leq^* g$; indeed, fix some large enough $\alpha\in \kappa\setm \sup_{n<\oo} \delta^{i_n}$ that also satisfies $\sup_{n<\oo}s_{A^{i_n}}(\alpha)<g(\alpha)$. We will show that $f(\alpha)<g(\alpha)$. 

There is an $n<\oo$ so that $\sup (E^{i_n}_1\cap (\alpha+1))=\sup (E^{i_{n+1}}_1\cap(\alpha+1))$ and let $\xi$ denote this common value. Then $\xi\in E^{i_n}_2$ as well and $$\xi\leq \alpha\leq f(\alpha)<s_{E^{i_n}_1}(\xi)<s_{A^{i_n}}(\alpha)< g(\alpha)$$ as desired.
\end{proof}

%

\medskip

\subsection{On cofinalities} A particular case of Theorem \ref{newbound} addresses \cite[Question 15]{dilip2}:

\begin{cor}\label{cor:alephn}
$\mf d(\aleph_n)\leq \mf r(\aleph_n)$ unless  $\aleph_{\oo_n}\leq \mf r(\aleph_n)$ for $n<\oo$. 
\end{cor}

Furthermore, one can use  Theorem \ref{newbound}  with PCF-theoretic bounds for cofinalities to relate $\mf d(\kappa)$ and $\mf r(\kappa)$: 

\begin{cor}\label{cor:pcf}
If $\mf r(\aleph_1)<\aleph_{\oo_2}$ then $\mf d(\aleph_1)<\aleph_{\oo_5}$.
\end{cor}
\begin{proof}
Recall that  $$\cf([\aleph_\delta]^{|\delta|})<\aleph_{|\delta|^{+4}}$$  for any $\delta<\aleph_\delta$  \cite[Theorem 7.2]{arithmetic}.  Suppose that $\mf r(\aleph_1)=\aleph_{\delta+n}$ for some limit $\omg\leq \delta<\oo_2$ and $n\in \oo$. In turn, by Theorem \ref{newbound} (\ref{two}) applied with $\theta=\aleph_1<\mf b(\aleph_1)$, $$\mf d(\aleph_1)\leq \cf([\aleph_{\delta+n}]^\omg)\leq \aleph_{\delta+n}\cdot \cf([\aleph_{\delta}]^\omg)<\aleph_{\oo_5}$$ as desired.
\end{proof}


\medskip

The above results point us to the following interesting question: what can the cofinality of $\mf r(\kappa)$ be? Indeed, it is famously open whether $\cf(\mf r)=\oo$ is consistent,\footnote{Here $\mf r,\mf d$  denote $\mf r(\aleph_0)$ and $\mf d(\aleph_0)$, respectively.}  however $\cf(\mf r)=\oo$ does imply $\mf d\leq \mf r$ \cite{shelah_cof}.


\begin{thm}\label{thm:cof}
  If $\cf(\mf r(\kappa))\leq \kappa$ then  $\mf d(\kappa)\leq \mf r(\kappa)$ for any uncountable, regular $\kappa$.
\end{thm}

We mention that our proof is very specific to the uncountable case (and does not use Raghavan and Shelah's recent work).

\begin{proof}Let us assume that $\mu=\mf r(\kappa)<\mf d(\kappa)$ and $\cf(\mu)\leq \kappa$. Given some $F\subs [\kappa]^\kappa$ of size $\mu$, we construct a set $B$ which splits each $A\in F$. 

 We can write $F$ as an increasing union $\cup\{F_\xi:\xi<\lambda\}$ where $\lambda=\cf(\mu)$, so that $|F_\xi|<\mu$ and find $B_\xi$ that splits each $A\in F_\xi$. Our job is to glue together the sets $\{B_\xi:\xi<\lambda\}$ into a single set $B$ which splits all $A\in F$ at once.

Let us consider the $\lambda=\kappa$ case first: we would like to find a fast enough club $E\subs \kappa$ so that $$B=\bigcup\{[\xi,s_E(\xi))\cap B_\xi:\xi\in E\}$$ works. So, take an elementary submodel $M\prec H(\Theta)$ of size $\mu$ such that $F\cup \{B_\alpha:\alpha<\kappa\}\subs M$. By assumption $M\cap \kappa^\kappa$ is not dominating, so we can find a single $f\in \kappa^\kappa$ so that for any $g\in M\cap \kappa^\kappa$, the set $$I_g=\{\alpha<\kappa:g(\alpha)<f(\alpha)\}$$ has size $\kappa$. Now let $E$ be a club of $f$-closed ordinals in $\kappa$, and we claim that the above defined $B$ splits each $A\in F$.

Fix some $A\in F$ such that $A\in F_{\zeta}$ and note that $$g(\alpha)=\sup\{\min(A\cap B_{\alpha'}\setm (\alpha+1)):\alpha'\leq \alpha\}$$ is well defined for  $\alpha\geq \zeta$, and $g(\alpha)<\kappa$. The crucial property of $g$ we use is that $$(\alpha,g(\alpha))\cap A\cap B_{\alpha'}\neq \emptyset$$ for any $\alpha'\leq \alpha$. Since $g\in M$, the set $I_g$ has size $\kappa$.

Now, given $\alpha\in I_g\setm \zeta$, find the interval from $E$ that contains $\alpha$: let $\xi_\alpha=\sup(E\cap (\alpha+1))$ and note that $$\xi_\alpha\leq \alpha<g(\alpha)<f(\alpha)<s_E(\xi_\alpha)$$ since $E$ was $f$-closed. In particular, $$[\xi_\alpha,s_E(\xi_\alpha))\cap A\cap B=[\xi_\alpha,s_E(\xi_\alpha))\cap A\cap B_{\xi_\alpha}\neq \emptyset$$ and so $|B\cap A|\geq |\{\xi_\alpha:\alpha\in I_g\setm \zeta \}|= \kappa$ as desired.

We can similarly prove $|A\setm B|=\kappa$ by looking at the function $$g'(\alpha)=\sup\{\min (A\setm (B_{\alpha'}\cup (\alpha+1))):\alpha'\leq \alpha\}$$ and noting that $$[\xi_\alpha,s_E(\xi_\alpha))\cap A\setm B=[\xi_\alpha,s_E(\xi_\alpha))\cap A\setm B_{\xi_\alpha}\neq \emptyset$$ for any $\alpha\in I_{g'}\setm \zeta$.

\medskip

Now, suppose $\lambda=\cf(\mu)<\kappa$: this case will be handled similarly although the arguments are a bit more involved. We take an increasing sequence of elementary submodels $(M_\nu)_{\nu<\lambda}$ and sequence of functions $(f_\nu)_{\nu<\lambda}$ such that 
\begin{enumerate}
\item each $M_\nu$ has size $\mu$, and $(M_\eta)_{\eta<\nu}\in M_\nu$,
\item $f_\nu\in \kappa^\kappa\cap  M_{\nu+1}$ and $\eta<\nu$ implies $f_\eta<f_\nu$, and
\item $f_\nu$ is not $\leq^*$-dominated by any $g\in \kappa^\kappa\cap M_\nu$.
\end{enumerate}
Construct clubs $E_\nu\subs \kappa$ and functions $s^\nu\in \kappa^\kappa$ for $\nu<\lambda$ as follows: $s^0$ is the identity function on $\kappa$, and $E_0$ is an $f_0$-closed club. In general, $s^{\nu+1}=s_{E_\nu}\circ s^\nu$ and $s^\nu=\sup_{\eta<\nu}s^\eta$ for limit $\nu<\lambda$. We pick $$E_\nu\subs \bigcap_{\eta<\nu}E_\eta$$ so that each $\xi\in E_\nu$ is closed under $f_\nu$ and $s_\nu$. Moreover, we pick each $E_\nu\in M_{\nu+1}$ canonically which ensures $s^{\nu+1}\in M_{\nu+1}$ as well.

Now, we construct a club $E=\{\xi_\gamma:\gamma<\kappa\}\subs \kappa$. Let us outline the first $\lambda$ steps and then describe the general construction: $\xi_0=0$, $\xi_1=s_{E_0}(\xi_0)$, $\xi_2=s_{E_1}(\xi_1)$ and so on. At limit steps $\nu$ we take supremum: $\xi_\nu=\sup_{\eta<\nu}\xi_\eta$, and let $\xi_{\nu+1}=s_{E_\nu}(\xi_\nu)$ in successor steps. In other words, $\xi_\nu=s^\nu(0)$ for $\nu<\lambda$. This defines the first $\lambda$ many elements of $E$. In general, any $\gamma<\kappa$ can be written as $\gamma=\lambda\cdot \gamma_0+\nu$ for a unique $\nu=\nu(\gamma)<\lambda$ and we set $$\xi_{\gamma+1}=s_{E_\nu}(\xi_\gamma)=s^{\nu}(\xi_{\lambda\cdot \gamma_0}).$$ 

\begin{figure}[H]

\includegraphics[width=0.6\textwidth]{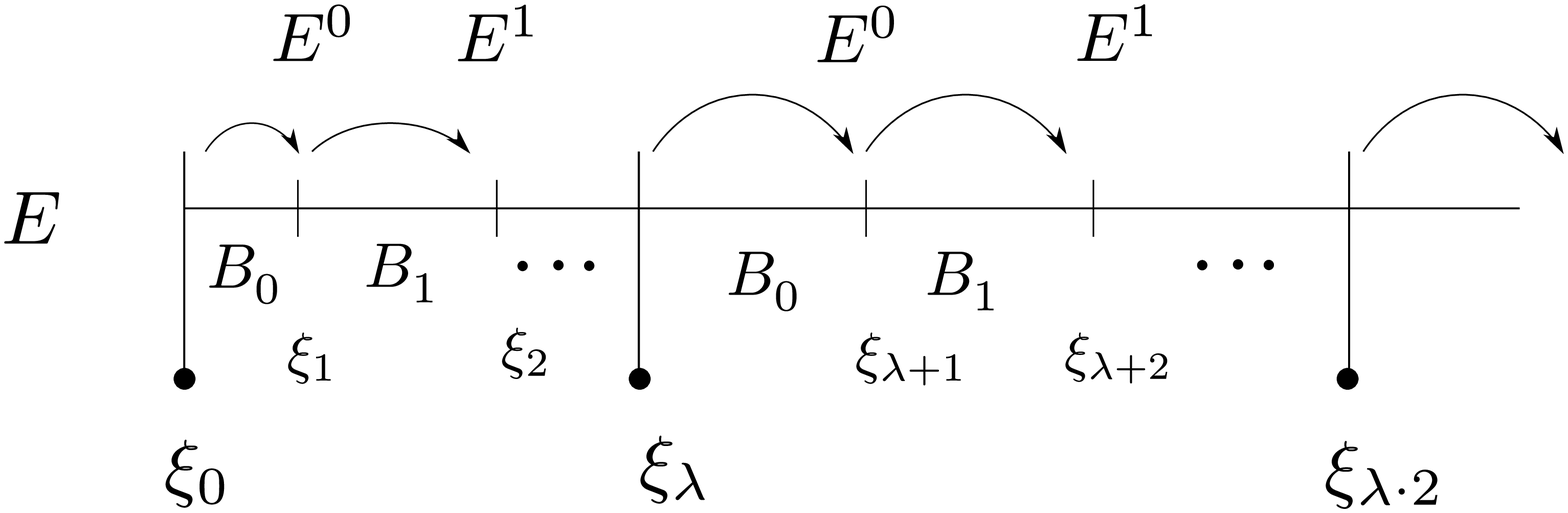}
  \centering
  \caption{Defining $B$ from $\{B_\xi:\xi<\lambda\}$}
  \label{cof_diag}
\end{figure}

This defines the club $E$, and we set $$B=\bigcup\{[\xi_\gamma,\xi_{\gamma+1})\cap B_{\nu(\gamma)}:\gamma<\kappa\}.$$

We would like to show that $B$ works: fix some $A\in F$, and we need that both $A\cap B$ and $A\setm B$ have size $\kappa$. 

Now, fix some $\nu_0<\lambda$ so that $A\in F_{\nu_0}$. Note that $A\cap B_\nu$ and $A\setm B_\nu$ are both of size $\kappa$ for $\nu_0\leq \nu<\lambda$ so the set of accumulation points $\acc(A\cap B_\nu)$ and $ \acc(A\setm B_\nu)$ are clubs in $\kappa$. In, particular we can define a function $g_A$ by setting \begin{equation}\label{eq}
g_A(\alpha)=\min \bigcap_{\nu\in\lambda\setm \nu_0}\bigl (\acc(A\cap B_\nu)\cap \acc(A\setm B_\nu)\bigr )\setm (s^{\nu_0+1}(\alpha)+1).
                                                                                                                                                                                                                                                           \end{equation}

The important fact here is that $g(\alpha)$ is an accumulation point of all the sets $A\cap B_\nu$ and $A\setm B_\nu$ (where $\nu\in \lambda\setm \nu_0$) above $s^{\nu_0+1}(\alpha)$. Moreover, since $g_A\in \kappa^\kappa\cap M_{\nu_0+1}$, the set $$I=\{\alpha<\kappa:g_A(\alpha)<f_{\nu_0+1}(\alpha)\}$$ must have size $\kappa$.

We can find a single $\nu<\lambda$ so that there are $\kappa$ many $\gamma<\kappa$ of the form $\gamma=\lambda\cdot \gamma_0+\nu$ such that $[\xi_\gamma,\xi_{\gamma+1})\cap I\neq \emptyset$.
Fix such an $\alpha\in  [\xi_\gamma,\xi_{\gamma+1})\cap I$ and recall that $\xi_{\gamma+1}\in E_\nu$. 

So if $\nu_0<\nu$ then  $\xi_{\gamma+1}$ is $s^{\nu_0+1}$-closed and so $s^{\nu_0+1}(\alpha)<\xi_{\gamma+1}$. In turn, 
\begin{equation}\label{eq2}
 \xi_\gamma\leq \alpha<g_A(\alpha)<f_{\nu_0+1}(\alpha)\leq f_\nu(\alpha)<\xi_{\gamma+1}
\end{equation}
since $\xi_{\gamma+1}$ was closed under $f_\nu$. Since $g_A(\alpha)$ was an accumulation point of both $A\cap B_{\nu(\gamma)}$ and $A\setm B_{\nu(\gamma)}$ (see the definition in (\ref{eq})) and by equation (\ref{eq2}), we must have $$[\xi_\gamma,\xi_{\gamma+1})\cap A\cap B_{\nu(\gamma)}\neq \emptyset,$$ and $$[\xi_\gamma,\xi_{\gamma+1})\cap A\setm B_{\nu(\gamma)}\neq \emptyset.$$

Now, consider the case when $\nu\leq \nu_0$. It still holds that $\alpha<s^\nu(\alpha)<\xi_{\gamma+1}$ since $\xi_{\gamma+1}$ is $s^\nu$-closed and so $s^{\nu+1}(\alpha)=s_{E_\nu}\circ s^\nu(\alpha) =\xi_{\gamma+1}$. In turn, $$s^{\nu_0+1}(\alpha)=\xi_{\lambda\cdot \gamma_0+\nu_0+1}=:\xi\in E_{\nu_0}.$$ The next point in our club $E$ is $\xi^+:=\xi_{\lambda\cdot \gamma_0+\nu_0+2}\in E_{\nu_0+1}$ which is $f_{\nu_0+1}$-closed. So $$\xi<g_A(\alpha)<f_{\nu_0+1}(\alpha)<\xi^+.$$ As $g_A(\alpha)$ is an accumulation point of $A\cap B_{\nu_0+1}$ and  $A\setm B_{\nu_0+1}$ above $\xi$, we see that both $$[\xi,\xi^+)\cap A\cap  B_{\nu_0+1}\neq \emptyset,$$ and $$[\xi,\xi^+)\cap A\setm B_{\nu_0+1}\neq \emptyset$$ holds.

All in all, we proved that there are $\kappa$ many $\gamma<\kappa$ so that both 

\begin{align*}
  [\xi_\gamma,\xi_{\gamma+1})\cap A\cap B=[\xi_\gamma,\xi_{\gamma+1})\cap A\cap B_{\nu(\gamma)}\neq \emptyset
\end{align*}
 and  
 \begin{align*}
   [\xi_\gamma,\xi_{\gamma+1})\cap A\setm B =[\xi_\gamma,\xi_{\gamma+1})\cap A\setm B_{\nu(\gamma)}\neq \emptyset
 \end{align*}
 
 holds. In turn, $B\cap A$ and $A\setm B$ both have size $\kappa$.



\end{proof}

Hence $\mf d(\aleph_1)\leq \mf r(\aleph_1)$ unless $\mf r(\aleph_1)> \aleph_{\oo_1}$. It is still open if the inequality $\mf r(\aleph_1)<2^{\aleph_1}$ is consistent, and hence we don't know whether $\cf{\mf r(\kappa)}\leq \kappa$ is consistent at all.

\medskip

The inequalities known to us are summarized in Figure \ref{diag_a0}-\ref{diag_bo} below, where arrows point to cardinals greater or equal. First, the most studied case when $\kappa$ is $\aleph_0$, with plenty of independence between the characteristics:

\begin{figure}[H]

\includegraphics[width=.75\textwidth]{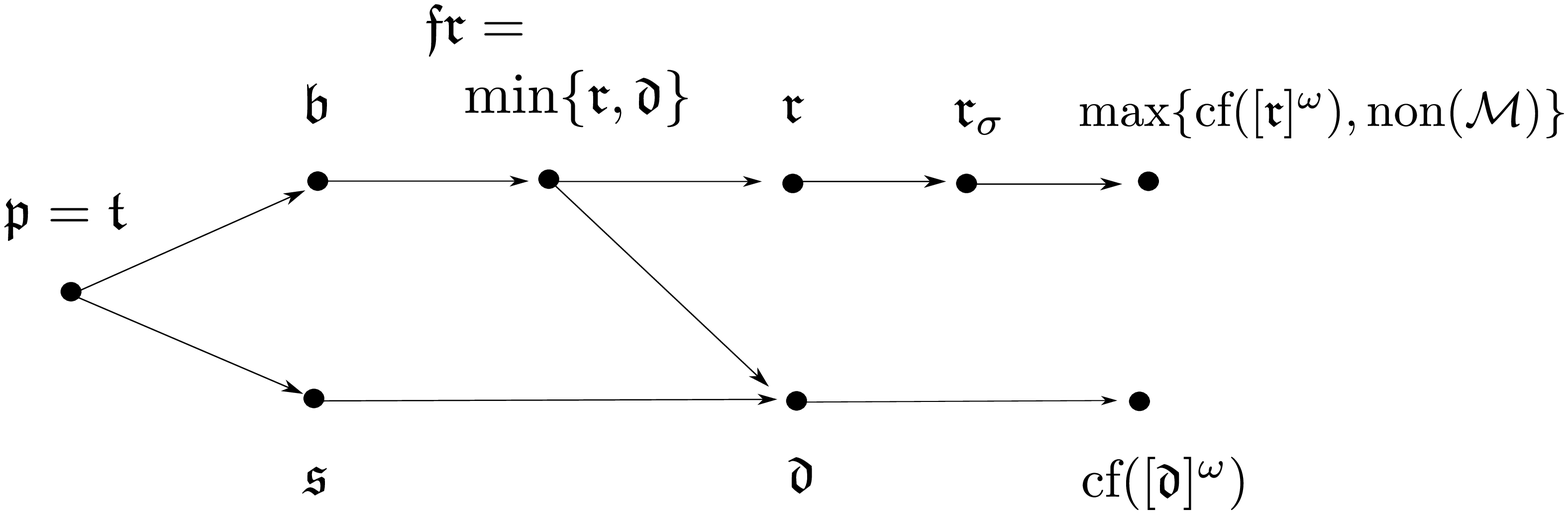}
  \centering
  \caption{The case of $\kappa=\aleph_0$}
  \label{diag_a0}
\end{figure}

\medskip

Next, note how the picture simplifies between the splitting and reaping numbers as we move to uncountable values of $\kappa$:
\medskip

\begin{figure}[H]

\includegraphics[width=1\textwidth]{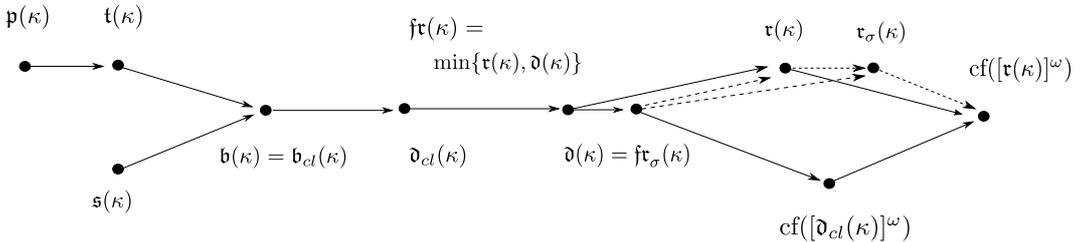}
  \centering
  \caption{The case of $\kappa=\cf(\kappa)>\aleph_0$}
  \label{diag_reg}
\end{figure}

The dashed arrows adjacent to $\mf r_\sigma(\kappa)$ hold whenever $\mf r_\sigma(\kappa)$ exists (that is, for $\kappa>2^{\aleph_0}$). Furthermore, the dashed arrow between $\mf r(\kappa)$ and $\mf d(\kappa)$ is valid also when $\cf(\mf r(\kappa))\leq \kappa$.

\medskip

Finally, an even more linear diagram above $\beth_\oo$:
\medskip

\begin{figure}[H]

\includegraphics[width=.8\textwidth]{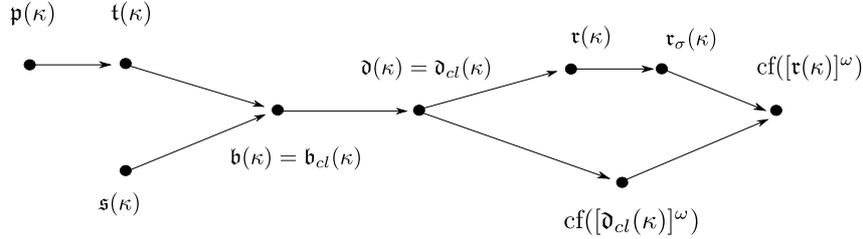}
  \centering
  \caption{The case of $\kappa=\cf(\kappa)\geq \beth_\oo$}
  \label{diag_bo}
\end{figure}

\section{Questions}

Finally, let us emphasize a few fascinating open problems about combinatorial cardinal characteristics; we would  also like to refer the interested reader to \cite{brook0, brook, khomskii} for further reading and questions on the generalized Baire space and Cichon's diagram.

\begin{prob}\cite{dilip}
 Is $\mf b(\kappa)<\mf a(\kappa)$ consistent for an uncountable, regular $\kappa$?
\end{prob}

The answer is yes at least for $\kappa=\aleph_0$ \cite{bashelah}.

\begin{prob}\cite{cummings}
 Does $\mf d(\kappa)=\mf d_{cl}(\kappa)$ hold for all uncountable, regular $\kappa$?
\end{prob}

\begin{prob}\cite{dilip2}
 Does $\mf d(\kappa)\leq \mf r(\kappa)$ hold for all uncountable, regular $\kappa$?
\end{prob}

Recall that if $\mf d(\kappa)=\mf d_{cl}(\kappa)$ then $\mf d(\kappa)\leq \mf r(\kappa)$  follows.

\begin{prob}
 Can $\mf r(\kappa)$ have cofinality at most $\kappa$? 
\end{prob}

In particular, is $\mf r(\aleph_1)=\aleph_\oo$ possible?

\begin{prob}\cite{dilip2}
 Is it consistent that $\mf r(\aleph_1)<2^{\aleph_1}$? 
\end{prob}

Given that $\mf s(\kappa)=\kappa$ unless $\kappa$ is quite large (i.e., weakly compact), one might conjecture a dual result: $\mf r(\kappa)<2^{\kappa}$ implies that $\kappa$ is large \cite{dilip2}.

\begin{prob}\cite{omer}
 Is it consistent that $\mf s(\kappa)$ is singular for some uncountable, regular $\kappa$?
\end{prob}

\begin{prob}
 Does $\mf p(\kappa)=\mf t(\kappa)$ for all  uncountable, regular $\kappa$?
\end{prob}

The last two problems have a positive answer for $\kappa=\aleph_0$ \cite{dow,malliaris}.


\begin{thebibliography}{2}
 
 \bibitem{arithmetic} Abraham, Uri, and Menachem Magidor. "Cardinal arithmetic." Handbook of set theory. Springer, Dordrecht, 2010. 1149-1227.

 \bibitem{omer} Ben-Neria, Omer, and Moti Gitik. "On the splitting number at regular cardinals." The Journal of Symbolic Logic 80.4 (2015): 1348-1360.
 
 \bibitem{char} Blass, Andreas. "Combinatorial cardinal characteristics of the continuum." Handbook of set theory. Springer, Dordrecht, 2010. 395-489.
 
\bibitem{blass} Blass, Andreas, Tapani Hyttinen, and Yi Zhang. "Mad families and their neighbors." preprint (2005).

\bibitem{brook0} Brooke-Taylor, A. D., Brendle, J., Friedman, S. D., \& Montoya, D. C. "Cichon's diagram for uncountable cardinals." Israel Journal of Mathematics (2017).

\bibitem{brook} Brooke-Taylor, Andrew D., Fischer, V., Friedman, S. D., \& Montoya, D. C. "Cardinal characteristics at $\kappa$ in a small $\mf u (\kappa)$ model." Annals of Pure and Applied Logic 168.1 (2017): 37-49.

\bibitem{brendlesplitting} Brendle, J\"org. "Around splitting and reaping." Commentationes Mathematicae Universitatis Carolinae 39.2 (1998): 269-279.


\bibitem{cummings} Cummings, James, and Saharon Shelah. "Cardinal invariants above the continuum." Annals of Pure and Applied Logic 75.3 (1995): 251-268.

\bibitem{dow} Dow, Alan, and Saharon Shelah. "On the cofinality of the splitting number." Indagationes Mathematicae 29.1 (2018): 382-395.

\bibitem{garti} Garti, Shimon. "Pity on lambda." arXiv preprint:1103.1947 (2011).

\bibitem{hytt} Hyttinen, Tapani. "Cardinal invariants and eventually different functions." Bulletin of the London Mathematical Society 38.1 (2006): 34-42.

\bibitem{khomskii} Khomskii, Y., Laguzzi, G., L\"owe, B., Sharankou, I.. Questions on generalised Baire spaces. Mathematical Logic Quarterly, 62(4-5), 439-456.

\bibitem{malliaris} Malliaris, Maryanthe, and Saharon Shelah. "General topology meets model theory, on $\mf p$ and $\mf t$." Proceedings of the National Academy of Sciences 110.33 (2013): 13300-13305.


\bibitem{dilip} Raghavan, Dilip, and Saharon Shelah. "Two results on cardinal invariants at uncountable cardinals." arXiv preprint arXiv:1801.09369 (2018).

\bibitem{dilip2} Raghavan, Dilip, and Saharon Shelah. "Two inequalities between cardinal invariants." arXiv preprint arXiv:1505.06296 (2015).

\bibitem{bashelah} Shelah, Saharon. "Two cardinal invariants of the continuum ($\mf b< \mf a$) and FS linearly ordered iterated forcing." Acta Mathematica 192.2 (2004): 187.

\bibitem{shelah_cof} Shelah, Saharon. "On reaping number having countable cofinality." arXiv preprint arXiv:1401.4649 (2014).

\bibitem{suzuki} Suzuki, Toshio. "About splitting numbers." Proceedings of the Japan Academy, Series A, Mathematical Sciences 74.2 (1998): 33-35.

\bibitem{zapletal} Zapletal, Jindrich. "Splitting number at uncountable cardinals." The Journal of Symbolic Logic 62.1 (1997): 35-42.

\end{thebibliography}
\end{document}